\documentclass[11pt,english,reqno]{amsart}

\usepackage[top=3cm,bottom=3.5cm,inner=3cm,outer=3cm]{geometry}
\usepackage{setspace}
\usepackage[pagebackref]{hyperref}
\usepackage{enumerate}
\usepackage{amsmath,amsfonts,amssymb,amsthm}
\usepackage[T1]{fontenc}
\usepackage{caption}
\usepackage{graphicx}
\usepackage{csquotes}
\usepackage{tikz} 
\usetikzlibrary{decorations.pathreplacing, calc, arrows, decorations.markings}

\MakeOuterQuote{"}

\newtheorem{theorem}{Theorem}
\newtheorem{corollary}[theorem]{Corollary}
\newtheorem{conjecture}[theorem]{Conjecture}
\newtheorem{lemma}[theorem]{Lemma}

\newcommand{\dir}{\mathrm{dir}\,}

\newcommand{\F}{\mathcal{F}}
\newcommand{\G}{\mathcal{G}}
\newcommand{\Gd}{\G_{\dir}}
\newcommand{\Go}{\G_o^{a,b}}
\begin{document}

\title{New $\F$-Saturation Games on Directed Graphs}

\author[J. D. Lee]{Jonathan D. Lee}
\address{Department of Pure Mathematics and Mathematical Statistics, University of Cambridge, Wilberforce Road, Cambridge CB3\thinspace0WB, UK}
\email{j.d.lee@dpmms.cam.ac.uk}

\author[A. Riet]{Ago-Erik Riet}
\address{University of Tartu, Faculty of Mathematics and Computer Science, Institute of Computer Science, Juhan Liivi 2, 50409 Tartu, Estonia}
\email{ago-erik.riet@ut.ee}

\date{\today}

\begin{abstract}
We study analogues of \emph{$\F$-saturation games}, first introduced by F\"uredi, Reimer and Seress~\cite{FuReSe} in 1991, and named as such by West~\cite{We}. We examine analogous games on directed graphs, and show tight results on the walk-avoiding game. We also examine an intermediate game played on undirected graphs, such that there exists an orientation avoiding a given family of directed graphs, and show bounds on the score. This last game is shown to be equivalent to a recent game studied by Hefetz, Krivelevich, Naor, and Stojakovi\'{c}in \cite{HeKrNaSt}, and we give new bounds for biased versions of this game.
\end{abstract}

\maketitle

\section{Introduction}

For $\F$ a family of digraphs, we say a digraph $G$ is \emph{$\F$-homomorphism-free} if for all $F \in \F$, there is no homomorphism $F \rightarrow G$. We say $G\subset H$ is a \emph{$\F$-homomorphism-saturated} subgraph of $H$ if $G$ is a maximal $\F$-homomorphism-free subgraph of $H$. Take $H$ a digraph, $|H| = n$, and let $\F$ be a family of digraphs. In a similar fashion to the definition of the triangle free game of F\"uredi, Reimer and Seress~\cite{FuReSe}, and building on the notation of West\cite{We}, we define the \emph{directed $\F$-homomorphism-saturation game} as follows.

We have two players, \emph{Prolonger} and \emph{Shortener}, who we take to be male and female respectively. We define a graph process $\G_i$. We initially set $\G_0 = E_n$, the empty graph on $n$ vertices. The process ends at time $t^*$ if $G_{t^*}$ is $\F$-homomorphism-saturated. Otherwise, at time $2t$ , Prolonger chooses an edge $uv$, such that $uv \not\in \G_{2t}$ and $\G_{2t}\cup uv$ is $\F$-homomorphism-free, and $\G_{2t+1} = G_{2t} \cup uv$. Similarly, at time $2t+1$ Shortener chooses an edge to add, such that the process remains $\F$-homomorphism-free. Prolonger's goal is to maximise $t^*$, whilst Shortener wishes to minimise $t^*$. Our results will not depend on which of the two players moves first, and so we refer to this game as $\Gd(H;\F)$. We say the value of $t^*$ under optimal play by both Prolonger and Shortener is the \emph{score} or \emph{game saturation number} of $\Gd(H;\F)$, denoted by $\textbf{Sat}(\Gd)$ or simply $\Gd$ provided there can be no confusion. When only one graph is excluded, we write $\Gd(H;F) := \Gd(H;\{F\})$.

For $\F$ a family of digraphs, we say a graph $G$ is \emph{$\F$-orientation-free} if there is an orientation $G'$ of $G$ such that $G'$ contains no $F \in \F$ as a subdigraph. We say $G$ is a \emph{$\F$-orientation-saturated} subgraph of $H$ if $G$ is a maximal $\F$-orientation-free subgraph of $H$. Take $H$ a graph, $|H| = n$, and let $\F$ be a family of graphs. We define the \emph{$\F$-orientation-saturation game} similarly to the directed $\F$-homomorphism-saturation game. The game terminates at time $t^*$ if $G_{t^*}$ is $\F$-orientation-saturated. Otherwise a player chooses a new edge from $H$ to add to the process such that the process remains $\F$-orientation-free. We \emph{bias} the game by allowing Prolonger to take $a$ consecutive turns, and then Shortener to take $b$ consecutive turns. We refer to this game as $\Go$, and carry over notation analogously.

\section{Our Results}
\begin{theorem}\label{directedgamescore}
Let $n, k$ be any integers. Then we have:
\[\Gd(K_n, P_k) = \left\{\begin{array}{c l}      
    0 & k \leq 2\\
    \left\lfloor\frac{n^2}{4}\right\rfloor & k = 3\\
\frac{1}{3}n^2 + \frac{1}{3}nk + O(n + k^2) & k \geq 4 \end{array}\right.\]
\end{theorem}
\begin{theorem}\label{orientedgamescore}
For all $a, b \ll k \ll n$, we have:
\[\binom{n}{2}\left(1 - \frac{1}{\lambda^- k}\right) (1+o(1))\leq \Go(K_n, P_{k+1}) \leq \binom{n}{2}\left(1 - \frac{1}{\lambda^+ k}\right) (1+o(1)),\]
with constants $\lambda^+, \lambda^-$ defined as:
\[\lambda^- = \frac{\left\lfloor\frac{b}{2a}\right\rfloor}{1+\left\lfloor\frac{b}{2a}\right\rfloor},\quad
  \lambda^+ = \frac{1}{1+\left\lfloor\frac{a}{2b}\right\rfloor}.\] 
\end{theorem}
We further conjecture that
\begin{conjecture}
For all $a, b \ll k \ll n$, we have:
\[\Go(K_n, P_{k+1}) = \binom{n}{2}\left(1 - \frac{1}{\lambda k}\right) (1+o(1)),\]
with
\[\lambda = \frac{2a}{b+2a}\]
\end{conjecture}

In section~\ref{directedsection}, we turn our attention to the walk-avoiding game $\Gd(K_n, P_k)$ and Theorem~\ref{directedgamescore}. In section~\ref{orientedsection} we will prove Theorem~\ref{orientedgamescore}.

\section{Avoiding a directed walk on \texorpdfstring{$k$}{k} vertices}\label{directedsection}

We begin with a structural classification of all $P_k$-homomorphism-saturated graphs. This lemma will enable us to reason about the ultimate score of the game from an early stage, and thus allow for the exhibition of effective strategies for both Prolonger and Shortener.

\begin{lemma}
Any directed $P_k$-homomorphism-saturated graph $G$ on at least $k-1$ vertices is induced by a total order on $k - 1$ vertex classes $G_1,\ldots,G_{k-1}$. If $|G| \geq k$ then $|G_i| > 0\;\forall\; i$, otherwise $|G_i| \leq 1$.
\end{lemma}
\begin{proof}
Since there is no graph homomorphism from $P_k$ to a $P_k$ saturated graph, $G$ must be acyclic. Hence $G$ can be topologically sorted to give a
linear ordering of the vertices such that $\overrightarrow{uv} \in E(G)$ implies $u < v$ in the ordering. Let us call $v_1v_2\ldots v_m$ a descending sequence of vertices if $v_iv_{i+1}$ is a directed edge for every $i = 1, 2,\ldots,m-1$. Suppose that the longest descending sequence of vertices has $l$ vertices. Let us define classes $G_1,G_2,\ldots,G_l$ so that $v \in G_i$ if the longest descending sequence of vertices
ending at $v$ has $i$ elements. Then for any edge $\overrightarrow{uv}$, $u \in G_i, v \in G_j$ for $i < j$.

By construction $|G_i| > 0 \Rightarrow |G_j| > 0 \;\forall\; j < i$. If $|G_{k-1}| = 0$ and $u, v \in G_i$ are distinct,
then $G \cup \overrightarrow{uv}$ has no $P_k$, so G is not $P_k$ saturated. Hence either all classes are non-empty, or all classes have size at most $1$, in which case $|G| < k-1$, and $G$ is induced by a total order on its vertices.

If $u \in G_i, v \in G_j, i < j$, then any walk in $G\cup \overrightarrow{uv}$ using $\overrightarrow{uv}$ has length at most
$i + 1 + (k - 1 - j) < k$, so $uv \in G$ if $G$ is $P_k$ saturated. Hence $G$ is induced by the total ordering of the vertex classes $G_i$.
\end{proof}

In the following discussion, we will define the \emph{vertex classes} of a $P_k$-homomorphism-saturated graph $G$, as $G_i$ the set of vertices $v$ such that the longest path in $G$ terminating at $v$ has $i - 1$ edges. Hence $V = G_1 \sqcup \ldots \sqcup G_{k-1}$.

\begin{corollary}\label{size-corol}
In a $P_k$-homomorphism-saturated graph $G$, with vertex classes $G_i$, we have \[e(G) = \frac{1}{2}\left(n^2 - \sum_{i=1}^{k-1}|G_i|^2\right).\]
\end{corollary}
\begin{proof}
As $n = \sum_i |G_i|$, we have that $e(G) = \sum_{i < j} |G_i||G_j| = \frac{1}{2}\left(n^2 - \sum_{i=1}^{k-1}|G_i|^2\right)$, as every pair of vertices are either in the same vertex class or have a directed edge between them.
\end{proof}

With this structural lemma in hand, and a characterisation of the score in terms of the size of the vertex classes, we now turn to explicit strategies for Shortener and Prolonger.

\begin{lemma}\label{path-Shortener} Let $n \geq k \geq 4$. Then:
\begin{align*}
\Gd(K_n, P_k) &\leq \frac{1}{2}\left(n^2 - k + 4 - 2\left\lfloor\frac{n-k+4}{3}\right\rfloor^2 - \left\lceil\frac{n-k+4}{3}\right\rceil^2\right)\\
&= \frac{1}{3}n^2 + \frac{1}{3}kn + O(n + k^2).
\end{align*}
\end{lemma}

\begin{proof}
We give a strategy for Shortener which enforces this upper bound. 

First let Shortener build a path of $k-1$ vertices. We proceed inductively on the number of vertices $l$ of a maximal path assembled so far, and let Shortener use the following strategy:
\begin{itemize}
\item if Prolonger puts down an isolated edge $\overrightarrow{uv}$ and the existing path $v_1 v_2 \ldots v_l$ has $k-3$ vertices or less then Shortener will add the edge $\overrightarrow{v_lu}$, thus forming the path $v_1 v_2 \ldots v_l u v$; 
\item if Prolonger puts down an isolated edge $uv$ and if the existing path has $k-2$ edges then Shortener will add the edge $\overrightarrow{v_{k-1}u}$ forming the path $v_1 v_2 \ldots v_{k-1} u v$ and forcing $v_{k-1}$ to be in $G_{k-1}$ or $G_{k-2}$; 
\item if Prolonger connects two vertices already on the longest path then Shortener will attach an isolated vertex at the end of the path; 
\item if Prolonger puts down an edge $\overrightarrow{v_i v}$ to an isolated vertex $v$ where $i<l$ then Shortener will create the edge $\overrightarrow{v_l v}$ thus extending the existing path to $v_1 v_2 \ldots v_l v$; 
\item if Prolonger puts down an edge $\overrightarrow{v v_i}$ from an isolated vertex $v$ where $1<i$ then Shortener will create the edge $\overrightarrow{v v_1}$ thus extending the existing path to $v v_1 v_2 \ldots v_l$; 
\item if Prolonger extends the existing path then Shortener will also just increase the existing longest path (if it has less than $k-1$ vertices) by attaching an isolated vertex at the end. 
\end{itemize}
Note that after the path of $k-1$ vertices is formed, there is one vertex in each class $G_i$ and there may be one additional vertex that could eventually belong to either $G_{k-2}$ or $G_{k-1}$.

Hence now we have a path $P = v_1 v_2 \ldots v_{k-1}$. Once this path exists, Shortener takes a new strategy:

Let $C$ be the component containing $P$. We will ensure that after Shortener's move, all vertices are in $C$ or are isolated. Let $v \in C$ and $u\notin C$. If Prolonger creates an edge $\overrightarrow{vu}$ then Shortener will create the edge $\overrightarrow{v_{k-2}u}$; if Prolonger creates an edge $\overrightarrow{uv}$ then Shortener will create the edge $\overrightarrow{uv_2}$. Hence Shortener forces $u$ to be in either $G_1$ or $G_{k-1}$. If Prolonger plays an edge internal to $C$ then if there is $v\notin C$ Shortener will create the edge $v_{k-2} v$, thus forcing $v$ into $G_{k-1}$. If Prolonger plays an isolated edge $uu'$ then Shortener will create the edge $v_{k-3} u$ thus putting $u,u' \in C$ and forcing $u$ into $G_{k-2}$ and $u'$ into $G_{k-1}$. See Figure~\ref{nooled3} for an illustration.

\begin{figure}[h]
\begin{tikzpicture}[decoration = {markings,
    mark = at position 0.5 with {\arrow{angle 90}}
  }
  ] 
    \tikzstyle{vertex}=[draw,circle,fill=blue,minimum size=0.2cm,inner sep=0pt]
    \node[vertex][label=above:$v_1$] (v1) at ( 1,0) {};
    \node[vertex][label=above:$v_2$] (v2) at ( 2,0) {};
    \node[vertex][label=above:$v_3$] (v3) at ( 3,0) {};
    \node[vertex][label=above:$v_4$] (v4) at ( 4,0) {};
    \node[vertex][label=above:$v_{k-4}$] (vk4) at (5,0) {};
    \node[vertex][label=above:$v_{k-3}$] (vk3) at (6,0) {};
    \node[vertex][label=above:$v_{k-2}$] (vk2) at (7,0) {};
    \node[vertex][label=above:$v_{k-1}$] (vk1) at (8,0) {};
    \node[vertex][label=below:$u_1$] (u1) at (3,-1) {};
    \node[vertex][label=below:$u_2$] (u2) at (4,-1) {};    
    \node[vertex][label=below:$u_3$] (u3) at (5,-1) {};
     \begin{scope}[every path/.style={-}]
       \draw[postaction = decorate]  (v1) -- (v2);
       \draw[postaction = decorate]  (v2) -- (v3);
       \draw[postaction = decorate]  (v3) -- (v4);
       \draw[postaction = decorate,dashed] (v4) -- (vk4);
       \draw[postaction = decorate]  (vk4) -- (vk3);
       \draw[postaction = decorate]  (vk3) -- (vk2);
       \draw[postaction = decorate]  (vk2) -- (vk1);
       \draw[postaction = decorate]  (u1) -- (v2);
       \draw[postaction = decorate]  (u1) -- (u2);
       \draw[postaction = decorate]  (vk3) -- (u2);
       \draw[postaction = decorate]  (u2) -- (u3);
     \end{scope}          
\end{tikzpicture}
\captionof{figure}{Shortener's strategy to force into classes the vertices of paths on 3 vertices}
\label{nooled3}
\end{figure}
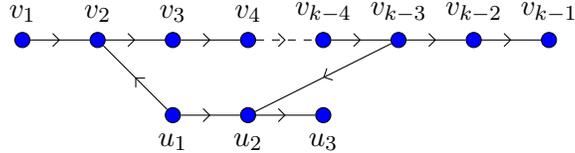

Hence after Shortener's move at most 1 vertex in $C$ does not have a fixed class, with the exceptional vertex in $G_{k-2}$ or $G_{k-1}$. Furthermore, all classes except $G_1$, $G_{k-2}$ and $G_{k-1}$ contain exactly one vertex. Given these conditions, by Corollary \ref{size-corol} the largest number of edges is achieved when each class $G_i$ where $2 \leq i \leq k-3$ contains one vertex and the classes $G_1$, $G_{k-2}$ and $G_{k-1}$ contain an almost equal number of vertices. Hence Shortener can force the game score to be at most $\frac{1}{2}(n^2-k+4-2\lfloor\frac{n-k+4}{3}\rfloor^2-\lceil\frac{n-k+4}{3}\rceil^2) = \frac{1}{3}n^2 + \frac{1}{3}kn + O(n + k^2)$.
\end{proof}

We now turn to an analysis of Prolonger's strategy.

\begin{lemma}\label{path-Prolonger}
Let $n \geq k \geq 4$. Then
\begin{align*}
\G_\dir & \geq \binom{k-1}{2} + (n-k+1)(k-2) + \frac{1}{2} (n-k-14)^2 (1 - \frac{1}{3})\\
& = \frac{1}{3}n^2 + \frac{1}{3}nk + O(n+k^2).\\
\end{align*}
\end{lemma}

\begin{proof}
By corollary~\ref{size-corol}, we know that the score is a concave funtion of the
sizes of the classes $G_i$. Hence to prove lower bounds on the game score it suffices to show that Prolonger can enforce constraints on the sizes of the classes $G_i$. Given an assignment of $r$ vertices to the $k - 1$
classes with $d_i$ vertices in class $i$, we define the \emph{normalised score} to be
$\sum_{i=1}^{k-1}(d_i/r)^2$. Note that the normalised score is a convex function of the density of the classes within the set of $r$ vertices.

Suppose we partition $G$ into disjoint subgraphs $D_1,\ldots,D_m$.
Each $D_j$ constrains the possible assignment of its vertices to the classes $G_i$, and we
denote by $s(D_j)$ the maximal normalised score of all possible assignments consistent with $D_j$. Note we have
\[
\sum_{j=1}^{m}s(D_j)\frac{|D_j|}{n} = \sum_{j=1}^{m}\sum_{i=1}^{k-1}\left(\frac{d_{ij}}{|D_j|}\right)^2 \frac{|D_j|}{n}
\]
for some of $d_{ij} \geq 0$ with $\sum_{i=1}^{k} d_{ij} = |D_j|$. Hence
\[
\sum_{j=1}^{m}s(D_j)\frac{|D_j|}{n} \geq \sum_{i=1}^{k-1}\left(\sum_{j=1}^{m}\frac{d_{ij}}{n}\right)^2 = \sum_{i=1}^{k-1}\frac{|G_i|^2}{n^2},
\]
and so $e(G) \geq \frac{n^2}{2}\left(1-\sum_i s(D_i)\frac{|D_i|}{n}\right)$.
Hence to prove that a strategy for Prolonger is effective it suffices to shows
that Prolonger can force a disjoint, spanning set of subgraphs of low
normalised score to appear in the game. We will refer to these subgraphs as
structures. In particular, we will show that Prolonger can force a path of
length $k - 1$ to exist, and force all but 14 of the remaining vertices to be in
structures of score at most $\frac{1}{3}$, which demonstrates the required bound.	

We now define a strategy for Prolonger to achieve the bound. As in
lemma~\ref{path-Shortener}, Prolonger can initially force a path of length $k - 1$ to exist,
with at most one additional vertex gaining degree greater than $0$. After this,
Prolonger will attempt to extend a maximal path. If the maximal path is of
length $< k - 2$, it is extended to absorb an isolated edge if possible,
and otherwise to a vertex of degree 0. If the path is of length $k - 2$, it is
extended to some other vertex in it's structure in preference to a vertex of
degree 0. We consider the following classes of structure (cf. Figure~\ref{abc}):

\begin{enumerate}[$A_\lambda$:]
\item a path on $\lambda+1$ vertices, which may be extended in either direction, and at most $\lambda$ non-isolated vertices.
\item a path on $\lambda+1$ vertices, with one vertex on the path having an additional in-edge or out-edge, which may prevent the extension of the path in one direction, and at most an additional $\lambda-1$ non-isolated vertices.
\item a path on $\lambda+1$ vertices, with vertices on the path with an additional in-edge and out-edge, which may prevent the extension of the path in both directions, and at most an additional $\lambda-2$-non isolated vertices. 
\end{enumerate}
We will abuse notation to speak of $s(A_\lambda) = \max \{s(G) : \exists F \in A_\lambda, F \textrm{ a spanning subgraph of }G\}$,
and we define $s(B_\lambda), s(C_\lambda)$ similarly.

\begin{figure}[h]
\scalebox{.45}{
 \begin{tikzpicture}[decoration = {markings,
    mark = at position 0.5 with {\arrow{angle 90}}
  }
  ] 
    \tikzstyle{vertex}=[draw,circle,fill=blue,minimum size=0.2cm,inner sep=0pt]
    \tikzstyle{cleanvertex}=[minimum size=0cm,inner sep=0pt]
    \node[vertex] (p1) at ( 0,5) {};
    \node[vertex] (p2) at ( 0,4) {};
    \node[vertex] (p3) at ( 0,3) {};
    \node[vertex] (p4) at ( 0,2) {};
    \node[vertex] (p5) at ( 0,1) {};
    \node[vertex] (p6) at ( 0,0) {};
    \node[vertex] (p7) at ( 1,0) {};
    \node[vertex] (p8) at ( 2,0) {};
    \node[vertex] (p9) at ( 3,0) {};
    \node[vertex] (p10)at( 4,0) {};
    \node[vertex] (p11)at( 5,0) {};
    \node[cleanvertex] (u1) at (2,1) {};
    \node[cleanvertex] (u2) at (3,1) {};
    \node[cleanvertex] (u3) at (4,1) {};
    \node[cleanvertex] (u4) at (5,1) {};
    \node[cleanvertex] (u5) at (6,1) {};

    \node[vertex] (q1) at ( 10,5) {};
    \node[vertex] (q2) at ( 10,4) {};
    \node[vertex] (q3) at ( 10,3) {};
    \node[vertex] (q4) at ( 10,2) {};
    \node[vertex] (q5) at ( 10,1) {};
    \node[vertex] (q6) at ( 10,0) {};
    \node[vertex] (q7) at ( 11,0) {};
    \node[vertex] (q8) at ( 12,0) {};
    \node[vertex] (q9) at ( 13,0) {};
    \node[vertex] (q10)at( 14,0) {};
    \node[cleanvertex] (v1) at (11,1) {};
    \node[cleanvertex] (v2) at (12,1) {};
    \node[cleanvertex] (v3) at (13,1) {};
    \node[cleanvertex] (v4) at (14,1) {};
    \node[cleanvertex] (v5) at (15,1) {};

    \node[vertex] (r1) at ( 20,5) {};
    \node[vertex] (r2) at ( 20,4) {};
    \node[vertex] (r3) at ( 20,3) {};
    \node[vertex] (r4) at ( 20,2) {};
    \node[vertex] (r5) at ( 20,1) {};
    \node[vertex] (r6) at ( 20,0) {};
    \node[vertex] (r7) at ( 21,0) {};
    \node[vertex] (r8) at ( 22,0) {};
    \node[vertex] (r9) at ( 23,0) {};
    \node[cleanvertex] (w1) at (21,4) {};
    \node[cleanvertex] (w2) at (21,1) {};
    \node[cleanvertex] (w3) at (22,1) {};
    \node[cleanvertex] (w4) at (23,1) {};
    \node[cleanvertex] (w5) at (24,1) {};
    
     \begin{scope}[every path/.style={-}]
       \draw[postaction = decorate]  (p1) -- (p2);
       \draw[postaction = decorate]  (p2) -- (p3);
       \draw[postaction = decorate,dashed] (p3) -- (p4);
       \draw[postaction = decorate]  (p4) -- (p5);
       \draw[postaction = decorate]  (p5) -- (p6);
       \draw[postaction = decorate]  (q1) -- (q2);
       \draw[postaction = decorate]  (q2) -- (q3);
       \draw[postaction = decorate,dashed] (q3) -- (q4);
       \draw[postaction = decorate]  (q4) -- (q5);
       \draw[postaction = decorate]  (q5) -- (q6);       
       \draw[postaction = decorate]  (r1) -- (r2);
       \draw[postaction = decorate]  (r2) -- (r3);
       \draw[postaction = decorate,dashed] (r3) -- (r4);
       \draw[postaction = decorate]  (r4) -- (r5);
       \draw[postaction = decorate]  (r5) -- (r6);       
       \draw[postaction = decorate]  (u1) -- (p7);
       \draw[postaction = decorate]  (u2) -- (p8);
       \draw[postaction = decorate]  (u3) -- (p9);
       \draw[postaction = decorate]  (u4) -- (p10);
       \draw[postaction = decorate]  (u5) -- (p11);
       \draw[postaction = decorate]  (v1) -- (q6);
       \draw[postaction = decorate]  (v2) -- (q7);
       \draw[postaction = decorate]  (v3) -- (q8);
       \draw[postaction = decorate]  (v4) -- (q9);
       \draw[postaction = decorate]  (v5) -- (q10);       
       \draw[postaction = decorate]  (w1) -- (r1);
       \draw[postaction = decorate]  (w2) -- (r6);
       \draw[postaction = decorate]  (w3) -- (r7);
       \draw[postaction = decorate]  (w4) -- (r8);
       \draw[postaction = decorate]  (w5) -- (r9);              
       \draw[dashed]  (p8) -- (p9);              
       \draw[dashed]  (q8) -- (q9);              
       \draw[dashed]  (r8) -- (r9);
\draw [decorate,decoration={brace,mirror,amplitude=10pt},xshift=-4pt,yshift=0pt]
(-.3,5) -- (-.3,0) node [black,midway,xshift=-1.2cm] 
{\LARGE $\lambda +1$};
\draw [decorate,decoration={brace,mirror,amplitude=10pt},xshift=-4pt,yshift=0pt]
(9.7,5) -- (9.7,0) node [black,midway,xshift=-1.2cm] 
{\LARGE $\lambda +1$};
\draw [decorate,decoration={brace,mirror,amplitude=10pt},xshift=-4pt,yshift=0pt]
(19.7,5) -- (19.7,0) node [black,midway,xshift=-1.2cm] 
{\LARGE $\lambda +1$};
\draw [decorate,decoration={brace,amplitude=10pt,mirror,raise=4pt},yshift=0pt]
(1,-.3) -- (5,-.3) node [black,midway,yshift=-1cm] {\LARGE $\lambda$};       
\draw [decorate,decoration={brace,amplitude=10pt,mirror,raise=4pt},yshift=0pt]
(11,-.3) -- (14,-.3) node [black,midway,yshift=-1cm] {\LARGE $\lambda -1$};       
\draw [decorate,decoration={brace,amplitude=10pt,mirror,raise=4pt},yshift=0pt]
(21,-.3) -- (23,-.3) node [black,midway,yshift=-1cm] {\LARGE $\lambda -2$};       
     \end{scope}          
 \end{tikzpicture}
}
\captionof{figure}{Structures $A_\lambda$, $B_\lambda$ and $C_\lambda$ respectively}
\label{abc}
\end{figure}
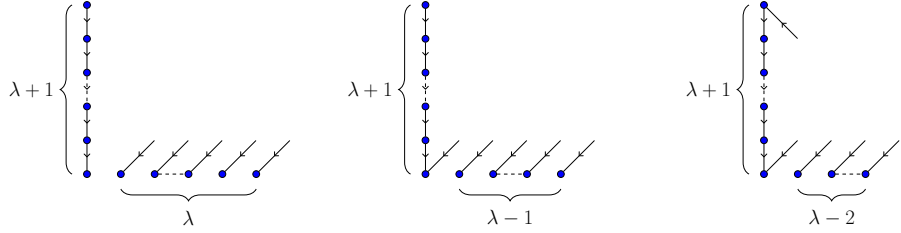

In any G containing $A_\lambda, B_\lambda$ or $C_\lambda$ as spanning subgraphs, there is at
least 1 vertex in $\lambda + 1$ classes. Hence if Shortener is permitted to control
the placement of the remaining vertices, she maximises the normalised score
of the structure by placing all of the remaining vertices in a single class
with one vertex of the path, and having the number of vertices off the path
being maximal. Hence we have:
\[s(A_\lambda)= \frac{\lambda+(\lambda+1)^2}{(2\lambda+1)^2},\quad
s(B_\lambda)= \frac{\lambda+\lambda^2}{(2\lambda)^2} = \frac{1}{4} + \frac{1}{4\lambda},\quad
s(C_\lambda)=\frac{\lambda+(\lambda-1)^2}{(2\lambda-1)^2}.
\]
These functions are decreasing in $\lambda$, and we have $s(A_6), s(B_3), s(C_2) \leq \frac{1}{3}$. We consider moves in pairs where Prolonger moves first, discarding at most 2 vertices touched by Shortener. To prevent Prolonger from increasing $\lambda$ for the structure he is building there either need to be no remaining vertices, or Shortener must prevent the path from growing. Assuming there are suitable vertices, Prolonger will extend the path, and so if Shortener spends her move to prevent the path from being extended in one direction the structure is changed from $A_\lambda$ into $B_{\geq \lambda+1}$ or $B_\lambda$ into $C_{\geq \lambda+1}$. If Shortener spends her move to make a new vertex $v$ have degree $> 0$, then Prolonger adds $v$ to the structure and so $\lambda$ has merely increased by 1. If Shortener plays an independent edge, then Prolonger either immediately incorporates it into a path, increasing $\lambda$ by 2, or only one vertex can be used in the path, which changes the structure from $A_\lambda$ into $B_{\lambda+1}$ or $B_\lambda$ into $C_{\lambda+1}$. Since a 1-point subgraph has structure $A_0$, we have that $C_2, B_{\geq3}$ or $A_{\geq6}$ is achieved as the structure of the subgraph unless Prolonger runs out of vertices of degree 0.

Note that since paths of length $k - 2$ are extended to other vertices in the structure, we have that for $4 \leq k < 6$ Prolonger will form $B_{k-1}$ structures in preference to $A_{k-1}$ structures, and so the normalised score of each structure will still be $\leq \frac{1}{3}$ for small $k$. We have $|A_6|, |B_2|, |C_1| \leq 11$, and control of at most one vertex was lost whilst producing the path of length $k - 1$, and we may discard 2 vertices to allow Prolonger to move first in the pairs, so there are at most 14 vertices that are not on structures of normalised score $\frac{1}{3}$ or the path of length $k-1$. Hence we have:
\[
\Gd(K_n, P_k) \geq \binom{k-1}{2} + (n-k-1)(k-2) + \frac{1}{2}(n-k-14)^2\left(1 - \frac{1}{3}\right)
\]
by considering pairs of vertices both on the path, pairs of vertices with exactly one on the path, and pairs of vertices disjoint from the path for the three terms.
\end{proof}
Having exhibited strategies for Prolonger and Shortener, we are now in a position to prove Theorem~\ref{directedgamescore}.
\begin{proof}[Proof of Theorem~\ref{directedgamescore}]
If $k \leq 2$, no move is legal and the score is 0.

If $k = 3$ then all graphs must be bipartite with classes $G_1,G_2$, and all edges
directed from $G_1$ to $G_2$. Hence any saturated graph will be complete
bipartite. Prolonger can ensure that the sets $G_1$, $G_2$ differ in size by at most 1.
Take $G_1$ to be the vertices of positive out-degree, and $G_2$ the
vertices of positive indegree. Initially $|G_1| = |G_2| = 0$, and after the first
move $|G_1| = |G_2| = 1$. Whenever Shortener moves, at most one vertex is
added to each of $G_1, G_2$. If $|G_1| \neq |G_2|$ and there is a vertex $v$ of degree 0,
then Prolonger can trivially add it to the smaller class. If $|G_1| = |G_2|$ and
there are two vertices $u, v$ of degree 0, then Prolonger can add the edge $\overrightarrow{uv}$.
Hence after Prolonger's move, either $|G_1| = |G_2|$ or $||G_1| - |G_2|| = 1$. Hence the score will be $\left\lfloor\frac{n^2}{4}\right\rfloor$.

For $k \geq 4$ we have by lemma~\ref{path-Shortener}:
\begin{align*}
\Gd(K_n, \{P_k\}) &\leq \frac{1}{2}\left(n^2 - k + 4 - 2\left\lfloor\frac{n-k+4}{3}\right\rfloor^2 - \left\lceil\frac{n-k+4}{3}\right\rceil^2\right)\\
&= \frac{1}{3}n^2 + \frac{1}{3}nk + O(n + k^2)
\end{align*}
Furthermore, we have by lemma~\ref{path-Prolonger}	:
\begin{align*}
\Gd(K_n, \{P_k\}) &\geq \binom{k-1}{2} + (n-k-1)(k-2) + \frac{1}{2}(n-k-12)^2\left(1 - \frac{1}{3}\right)\\
& = \frac{1}{3}n^2 + \frac{1}{3}nk + O(n + k^2)
\end{align*}
as required.
\end{proof}

\section{Games on undirected graphs derived from directed structures}\label{orientedsection}

In this section, we turn our attention to the orientation game. In the directed game, we forbade directed graphs
containing homomorphic images of a specified collection of directed graphs.
As a corollary, the direction of each edge was specified when it was picked
by Prolonger or Shortener. By construction, in this game this constraint is relaxed.

In particular, we consider $\Go(K_n, P_{k+1})$. Here, the Gallai-Hasse-Roy-Vitaver theoem~\cite{Ga}\cite{Has}\cite{Ro}\cite{Vi} states that the existence of a suitable orientation for $\G_i \Leftrightarrow \chi(\G_i) \leq k$, for $\chi$ the chromatic number. Equivalently, there exists a homomorphism $\G_i \rightarrow K_k$. Similar characterisations of somewhat larger classes of excluded directed graphs exist by other Gallai-Hasse type theorems.

A homomorphism $c : \G_i \rightarrow K_k$ implies that the sets $c^{-1}(i)$ are
independent sets of vertices for each $i \in [k]$. This is analogous to the situation in
$\Gd$, where we could define a function $c(v)$ the length of the longest path in $\G_i$ ending at $v$, with $\overrightarrow{uv} \in \G_i$ implied $c(u) < c(v)$. In $\G_o$ we only have $uv \in \G_i$ implies $c(u) \neq c(v)$.

As in the homomorphism game, the score is a concave quadratic function of the sizes of the classes, and so heuristically Shortener does well if many of the classes are very small, whilst Prolonger does well if the classes are all broadly equal in size. We conjecture that optimal play by Shortener will force some collection of vertices to be in classes of size 1, which requires that Shortener join each of these vertices to every other vertex. Conversely, Prolonger will conjecturally attempt to split vertices into many classes of equal size, by placing disjoint cliques. Insisting that each player takes an equal number of turns yields that Shortener forces $(1-\lambda)k$ vertices to be in isolated classes, whilst the remaining vertices are split into $\lambda k$ classes of approximately equal size by Prolonger placing $\sim\frac{n}{\lambda k}$ independent copies of $K_{\lambda k}$, with $\lambda = \frac{2a}{b+2a}$.

Heuristically, these bounds should be demonstrated by showing that as each player adds edges to vertices, the other can add appropriate edges to ensure that (i) no vertex is excluded from $(1-\lambda)k$ small classes before it is in an independent $K_{(\lambda - \epsilon)k}$, and that (ii) no vertex is in an independent $K_{(\lambda + \epsilon) k}$ before it is excluded from $(1-\lambda)k$ small classes. We cannot show bounds of this tightness, and instead demonstrate bounds on the value of $\lambda$.

\begin{theorem}
For all $a, b \ll k \ll n$, we have:
\[\binom{n}{2}\left(1 - \frac{1}{\lambda^- k}\right) (1+o(1))\leq \Go(K_n, P_{k+1}) \leq \binom{n}{2}\left(1 - \frac{1}{\lambda^+ k}\right) (1+o(1)),\]
with
\[\lambda^- = \frac{\left\lfloor\frac{b}{2a}\right\rfloor}{1+\left\lfloor\frac{b}{2a}\right\rfloor},\quad
  \lambda^+ = \frac{1}{1+\left\lfloor\frac{a}{2b}\right\rfloor}.\] 
\end{theorem}
\begin{proof}
To demonstrate the lower bound we exhibit a strategy for Prolonger.

Set $c = \left\lfloor\frac{a}{2b}\right\rfloor$. Then since $a \geq 2cb$, Prolonger may add $c$ edges for each end of an edge placed on Shortener's turn. Prolonger also splits the vertices into sets of size $k\lambda^- a - 1$ or $k\lambda^- a - 2$, and considers any edge within one of these sets to be Red, whilst any edge between two sets is Blue.

Prolonger will enforce the constraint that after his move, for every vertex $v$ either all incident Red edges are present or $d_{Red}(v) \geq c d_{Blue}(v)$. We proceed by induction. Shortener's move adds $b$ edges, so $c \sum_v d_{Red}(v)$ increases by at most $2bc \leq a$. Take any $v$ such that $d_{Blue}(v) < c d_{Red}(v)$ after Shortener's move. By induction, all of the earlier incident Blue edges had $c$ incident Red edges added in response, so if $d(v) = k-1$ after Shortener's move, then $d_{Red}(v) \geq k\lambda^- - a - 1$. Hence all vertices can have Red edges added until either there are no remaining potential Red edges incident on the vertex or $d_{Red }(v) \geq c d_{Blue}(v)$.

If Prolonger completes this phase and has remaining edges to add, then he adds them as Red edges if possible. If it is not, then all Blue edges exist, and so the graph contains a disjoint spanning set of cliques, each of size at least $k\lambda^- - a - 2$. Hence the final score will be at least $\binom{n}{2}\left(1 - \frac{1}{\lambda^- k}\right) (1+o(1))$ as required (since $a = o(k)$).

To demonstrate the upper bound we exhibit a strategy for Shortener.

First, Shortener picks a set $S$ of $(1 - \lambda^+)k - b - 1$ vertices and build a complete graph on them. We discard the $O(k^2)$ vertices that Prolonger can touch with an edge whilst this is done. Set $c' = \left\lfloor\frac{b}{2a}\right\rfloor$.

Shortener considers edges into $S$ as Red, and all other edges Blue. Shortener will enforce the constraint that after Shortener has moved, for any vertex $v\notin S, d_{Blue}(v) \geq c'd_{Red}(v)$. We proceed by induction as in the earlier case. Any remaining edges to be added by Shortener will be arbitrary Blue edges if possible. If not, then all of the Red edges exist, and so each vertex in $S$ is in a singleton colour class. Given this, the largest possible score is achieved when the remaining vertices are distributed equally across $\lambda^+ k + b + 1 = \lambda^+ k (1+o(1))$ colour classes. Hence the final score is at most $\binom{n}{2}\left(1 - \frac{1}{\lambda^+ k}\right) (1+o(1))$ as required.
\end{proof}

The primary source of weakness in this proof strategy is a lack of good characterisation of how to allow the constraints to be slightly broken without fatally wounding the proof. Roughly speaking, allowing the constraints to be broken slightly should permit moves to be strategically chosen to do double duty, with both ends of an added edge restoring the desired constraints, and allowing the targetted ratio to be non-integral.

This weakness especially damages Prolonger's strategy, as each of Prolonger's edges increase the degree of two vertices in a useful fashion, whilst Shortener's moves only use one end of each edge. Generally, the problem seems to be that if (say) Shortener is able to violate a constraint slightly in a great many places, then by only working with these vertices and repeating the process required to violate the constraints it she can produce a vertex which violates the constraint by $O(\log n)$. 

A secondary source of weakness for the bound given by Prolonger's strategy is that the most difficult case to handle has Shortener adding many edges between the sets Prolonger is attempting to complete, rather than forcing the existence of small classes. As a result, we expect that optimal play by Shortener does not obstruct Prolonger so substantially. If these weaknesses were to be dealt with, the two bounds come together at our conjectured true behaviour. It is also notable that whilst the purely directed game has a score which is essentially independent of $k$, this game does not, and so reflects the expected behaviour of undirected games rather more closely. We also note that earlier work in \cite{HeKrNaSt} gives at best a $(1-O(\log k / k))\binom{n}{2}$ lower bound for the unbiased game.	

\section{Concluding Remarks}
In this paper, we show some first results extending the $\F$-saturation game of F\"uredi, Reimer and Seress~\cite{FuReSe} to directed graphs, and study a new orientation-saturation game on undirected graphs. The behaviours of these games show subtle dependence on the sizes of the graphs to be excluded, and critical dependence on characterisations of the saturated graphs.


\begin{thebibliography}{99}
\bibitem{BiHoWi} Bir\'o, Cs., Horn, P. and Wildstrom, D.J., On Hajnal's triangle free game, {\it slides for the conference Infinite and finite sets, June 13-17, 2011, Budapest}.

\bibitem{FuReSe} F\"uredi, Z., Reimer, D. and Seress, A., Triangle-Free Game and Extremal Graph Problems, {\it Congr. Numer.} {\bf 82} (1991), 123--128. 

\bibitem{Ga} Gallai, T., On directed graphs and circuits, Theory of Graphs (Proceedings of the Colloquium Tihany 1966), (1968), pp. 115-118.

\bibitem{Has} Hasse, M., Zur algebraischen Begr\"{u}ndung der Graphentheorie. I, Mathematische Nachrichten 28 (1965) (5-6): 275?290, doi:10.1002/mana.19650280503, MR 0179105.

\bibitem{HeKrNaSt} Hefetz, D., Krivelevich, M., Naor, A., Stojakovi\'{c}, M., On Saturation Games, preprint, arXiv:1406.2111v2 [math.CO]

\bibitem{Ro} Roy, B., Nombre chromatique et plus longs chemins d'un graphe, Rev. Fran�aise Informat. Recherche Op\'{e}rationnelle 1 (1967) (5): 129-132, MR 0225683.

\bibitem{Vi} Vitaver, L. M., Determination of minimal coloring of vertices of a graph by means of Boolean powers of the incidence matrix, Doklady Akademii Nauk SSSR, 147 (1962): 758-759, MR 0145509

\bibitem{We} West, D., The F-Saturation Game (2009) and Game Saturation Number (2011), {\it \url{http://www.math.uiuc.edu/~west/regs/fsatgame.html} } (last visited September 1, 2014).

\end{thebibliography}
\end{document}